\documentclass[leqno]{amsart}
\usepackage[utf8]{inputenc}
\usepackage[english]{babel}
\usepackage{newclude}
\usepackage{amsfonts,amssymb,amsmath,amsgen, amsthm}
\usepackage{color, xcolor}
\usepackage{pdfsync}
\usepackage{enumerate}
\usepackage{centernot}
\usepackage{comment}
\usepackage[a4paper, total={6in, 8in}]{geometry}
\usepackage{hyperref}

\newtheorem{theorem}{Theorem}[section]
\newtheorem{lem}[theorem]{Lemma}
\newtheorem{prop}[theorem]{Proposition}
\newtheorem{cor}[theorem]{Corollary}
\newtheorem{remark}[theorem]{Remark}
\theoremstyle{definition}

\def\R{\mathbb R}

\def\N{\mathbb N}
\def\d{\frac{d}{dt}}

\def\eps{\varepsilon}

\title[Constrained parabolic equations]{Global solutions and Asymptotic behavior to a norm-preserving non-local parabolic flow}

\author[B. Shakarov]{Boris Shakarov}
\address{Institut de math\'ematiques de Toulouse, 118, route de Narbonne, Toulouse, France}
\email{boris.shakarov@math.univ-toulouse.fr}

\date{\today}

\begin{document}
	\maketitle
	\begin{abstract}
We consider a nonlinear parabolic model that forces solutions to stay on a $L^2$-sphere through a nonlocal term in the equation. 
We study the local and global well-posedness on a bounded domain and the whole Euclidean space in the energy space. 
Then, we consider the solutions' asymptotic behavior. We prove strong convergence to a stationary state and asymptotic convergence to the ground state in bounded domains when the initial condition is positive.
	\end{abstract}


\section{Introduction}

In this article, we consider the following nonlinear, nonlocal parabolic flow
\begin{equation}\label{eq: flow}
    \begin{cases}
        \partial_t u = \Delta u - \omega u + \mu[u] |u|^{2\sigma}u, \\
        u_{| \partial \Omega} = 0, \\
        u(0) = u_0 \in H^1_0(\Omega),
    \end{cases}
\end{equation}
where $u:  \R^d \times \R \to \R$, $\omega \in \R$, $\Omega \subset \R^d$ is a bounded domain with $C^2$ boundary or $\Omega \equiv \R^d$, $0 \leq \sigma$, $\sigma < \frac{2}{d-2}$ when $ d \geq 3$, and $\sigma < \infty$ for $d \leq 2$. When $\Omega = \R^d$, the Dirichlet boundary condition is intended as $ \lim_{|x| \to \infty} u(x) = 0$. The functional $\mu$ is defined so that the $L^2$-norm of solutions remains constant in time, and it is given by
\begin{equation}\label{eq: mu}
    \mu[u(t)] = \frac{\| \nabla u(t)\|_{L^2}^2 + \omega \| u(t) \|_{L^2}^2}{\| u(t) \|_{L^{2\sigma + 2}}^{2\sigma + 2}}.
\end{equation}

Similar types of nonlocal heat flows arise in geometrical problems where the dynamics must preserve some $L^p$-norms \cite{Au98}, \cite{St00}. Moreover,  parabolic equations under the fixed $L^2$ norm constraints play a pivotal role in studying stationary states \cite{ChSuTo00, Du02, BaDu04, BeDuCo20} with numerical approaches. A similar model, already studied in \cite{AnCaSh22}, is widely used for numerical simulations and is (the discretization of) the constrained gradient flow 
\begin{equation}\label{eqHeat2}
    \partial_t u = \Delta u + \omega[u] u + \beta |u|^{2\sigma}u, \quad \omega[u] = \frac{\| \nabla u\|_{L^2}^2 - \beta \| u \|_{L^{2\sigma + 2}}^{2\sigma + 2}}{{\| u\|_{L^2}^2}},
\end{equation}
see also \cite{BeDuCo20, MaCh09, CaLi09}. In \cite{AnCaSh22}, the authors showed that in the case $\beta > 0$ and $\sigma >\frac{2}{d}$, global well-posedness and uniform boundedness of solutions do not hold for any initial data. In comparison, for model \eqref{eq: flow},  we show that a uniform bound for the $H^1$ norm of solutions holds independently on initial data or $\sigma < \frac{2}{d-2}$. Thus this model allows to consider a broader range of nonlinearities.  The uniform bound is further exploited to show asymptotic convergence to stationary states independently of the initial data. 

In \cite{MaCh13}, the equation \eqref{eq: flow} with $\omega = 0$ was examined on compact manifolds. In that work, local existence is shown from the same iterative approach as in \cite{CaLi09}.  Our work establishes local well-posedness in any bounded domain and for any $\omega \in \mathbb{R}$ through a different approach, namely the Schauder fixed point argument. We will also consider the case $\Omega = \R^d$. In this case, the contraction principle is employed to show local well-posedness for any $\sigma$ for $d \leq 4$. For $d \geq 4$, we need to assume $0 \leq \sigma < \frac{1}{(d-2)^+}$ or $\frac{2}{d} \leq \sigma < \frac{2}{(d-2)^+}$ due to an unusual mixed type nonlinearity, see Section \ref{secExist} for details. We use the notation 
\begin{equation}\label{eqd-2Plus}
    \frac{1}{(d-2)^+} = \begin{cases}
        \infty \ &\mbox{ if } d \leq 2, \\
        \frac{1}{d-2} \ &\mbox{ if } d \geq 3.
    \end{cases}
\end{equation}

The first result of this work is the following.

\begin{theorem}\label{thm: existence}
Let $\omega \in \R$,  $\Omega$ be a bounded domain with $C^2$ boundaries or $\Omega \equiv \R^d$ and $0 \leq \sigma < \frac{2}{(d-2)^+}$. If $\Omega \equiv \R^d$, then let $0 \leq \sigma < \frac{1}{(d-2)^+}$ or $\frac{2}{d} \leq \sigma < \frac{2}{(d-2)^+}$.  For any $u_0 \in H^1_0(\Omega)$, there exists $T_{max} >0$ and a  solution $u\in C\big([0,T_{max}), H^1_0(\Omega) \big)$ to \eqref{eq: flow}. Moreover, either $T_{max} = \infty$ or $T_{max} < \infty$  and 
\begin{equation*}
    \lim_{t \to T_{max}} \| \nabla u(t) \|_{L^2} = \infty.
\end{equation*}
\end{theorem}

We will proceed by studying the global well-posedness of model \eqref{eq: flow}. We show that solutions are always global in time, independently of (a subcritical) $\sigma$ and $\omega$, and the $H^1$ norm remains uniformly bounded in time.  This is in contrast with the model \eqref{eqHeat2} which allows growing up solutions as shown in \cite{AnCaSh22} for the intercritical regime $\frac{2}{d} \leq \sigma < \frac{2}{(d-2)^+}$. We will obtain this result by showing that 
\begin{equation}\label{eqFIntr}
     F[u(t)] = \frac{\| \nabla u(t)\|_{L^2}^2 + \omega \| u(t) \|_{L^2}^2}{\| u(t) \|_{L^{2\sigma + 2}}^{2}}. 
\end{equation}
is a continuous and decreasing function of time.

\begin{theorem}
Under the hypothesis of Theorem \ref{thm: existence}, for any $u_0 \in H^1_0(\Omega)$, the corresponding solution $u$ to \eqref{eq: flow} is global in time, that is  $u\in C\big([0,\infty), H^1_0(\Omega) \big)$. 
\end{theorem}

Having established the existence of global solutions, our focus shifts to examining their asymptotic behavior. Given the Lyapunov functional \eqref{eqFIntr},
the asymptotic behavior can be studied with techniques similar to the gradient flow models (see e.g. \cite{Ha10}).  In particular, we show that the $\omega$-limit set
\begin{equation}\label{eq: omegaLimit}
    W(u_0) = \{ v \in H^1_0(\Omega)\, :  \, \exists (t_j)_{j \in \N} \subset \R^+, \, t_j \to \infty \mbox{ as } j\to \infty\, :  \, u(t_j) \to v \mbox { in } H^1_0(\Omega) \}.
\end{equation}
contains only stationary states. Stationary states are solutions to the elliptic problem  
\begin{equation}\label{eq: stst}
     0 = \Delta Q - \omega Q + \mu_Q|Q|^{2\sigma}Q,
\end{equation}
for some $\omega, \mu_Q \in \R$. The literature regarding solutions to \eqref{eq: stst} is extensive. Existence \cite{CaLi82}, uniqueness of positive solutions \cite{GiNiNi79, Kw89}, and the existence of infinitely many solutions \cite{BeLi83} have been proven. In particular, the set 
\begin{equation}\label{eq: statStates1}
    S_{u_0,\omega} = \{ Q \in H^1_0(\Omega) \, :  \, Q \mbox{ satisfies }  \omega Q = \Delta Q + \mu_Q |Q|^{2\sigma}Q, \, \| Q \|_{L^2} = \| u_0\|_{L^2} \}
\end{equation}
is nonempty and, in general, contains several elements for any $\omega >0$ and $u_0 \in H^1(\Omega)$.  
Moreover, minimizing the functional \eqref{eqFIntr} yields the grounds states \cite{BeWe10}, that is the unique strictly positive solution to \eqref{eq: stst}.
We will show the following.
\begin{theorem}
Under the hypothesis of Theorem \ref{thm: existence}, for any $\omega >0$ and $u_0 \in H^1_0(\Omega)$, we have $\emptyset \neq W(u_0) \subset S_{u_0,\omega}$. The same result holds for any $u_0 \in H^1_{rad}(\R^d)$, $d \geq 2$.
\end{theorem}
After establishing the asymptotic convergence to a stationary state on a subsequence of times, two fundamental questions emerge naturally. The first question pertains to the existence of criteria on the initial data that ensure the stationary state corresponds to the ground state, which is defined as the unique positive solution to equation \eqref{eq: stst} (up to space translations in $\R^d$). We denote the set of all ground states with fixed $L^2$ norm and satisfying \eqref{eq: stst} as $G_{u_0,\omega}$. By exploiting the maximum principle and uniqueness of the ground state, we show the following. 
\begin{theorem}
   Under the hypothesis of Theorem \ref{thm: existence}, let $\omega >0$ and $\Omega = \{x \in \R^d \, : \, |x| < R\}$ for some $R >0$ and $u_0 >0$. Then
    \begin{equation*}
        \lim_{t \to \infty} \| u(t) - Q_{gs}\|_{H^1} = 0.
    \end{equation*}
    where $Q_{gs}$ is the unique positive solution to \eqref{eq: stst} with $\| Q_{gs} \|_{L^2} = \| u_0 \|_{L^2}$. 
\end{theorem}
The second question revolves around determining whether the $\omega$-limit set, which characterizes the long-term behavior of the solution, is a singleton. This entails investigating the properties of the linearized operator around a stationary state and showing its dynamical stability. Regrettably, at present, we are unable to provide any definitive conclusions, leaving this as an unresolved issue. 

It is worth noting that a comparable analysis can be conducted on alternative models, such as the Navier-Stokes equations. For example, in \cite{CaPuRo09, brzezniak20182d}, the authors have explored the behavior of the two-dimensional system under different constraints. Furthermore, the constraint concerning the $L^2$ norm of the solution has been investigated in the context of the complex Ginzburg-Landau type equation, as discussed in \cite{AnSh23}. This equation serves to model the presence of a reservoir for the mass of a condensate.

The structure of this paper is outlined as follows: Section \ref{secPrelimin} introduces some pertinent preliminaries concerning the heat semigroup generated by the Laplacian on $L^2(\Omega)$ and the stationary states of system \eqref{eq: flow}. In Section \ref{secExist}, we establish the local and global well-posedness of the model. Subsequently, Section \ref{secAsympt} is dedicated to examining the asymptotic behavior.

\section{Preliminaries} \label{secPrelimin}
For any $p \geq 1 $, we denote by $p' \in \R$ the $
 1 = 1/{p} + 1/{p'} $ with $p' = \infty$ 
We use the symbol $\lesssim$ as in $a \lesssim b$ to intend that there exists a constant $C>0$, such that $a \leq Cb$.

We use the Gagliardo-Nirenberg inequality (see \cite[Theorem $1.3.7$]{Ca03}.
\begin{prop}
     For any $f \in H^1(\Omega)$ and any $0 \leq \sigma < \frac{2}{(d-2)^+}$, there exists $C_{GN}(d, \sigma, \Omega) >0$ such that 
    \begin{equation}\label{eq: GN}
        \| f \|_{L^{2\sigma + 2}} \leq C_{GN} \| f\|_{L^2}^{1 - \frac{d\sigma}{2\sigma + 2}} \| \nabla f \|_{L^2}^\frac{d\sigma}{2\sigma + 2}.
    \end{equation}
\end{prop}

In bounded domains, the following lower bound on the $L^{2\sigma + 2}$-norm holds by H\"older's inequality.

\begin{prop}
    Let $\Omega \subset \R^d$ be bounded. Then, for any $\sigma > 0$ and $u \in L^{2\sigma + 2}(\Omega)$, 
    \begin{equation}\label{eq: UnifBnPot}
        \| u \|_{L^{2\sigma + 2}} \geq \frac{\| u \|_{L^2}}{Vol(\Omega)^{\frac{\sigma}{2\sigma + 2}}}.
    \end{equation}
\end{prop}

\subsection{Heat semigroup}\label{sec: semig}

We recall the properties of the Dirichlet heat semigroup, see \cite{QuSo19} for more details. 
Let $\Omega$ is always a bounded domain with $C^2$ boundary or $\Omega \equiv \R^d$. We denote by $e^{t\Delta}$ the semigroup associated with the Dirichlet Laplacian on $L^2(\Omega)$ and $L^2(\R^d)$. The function $u = e^{t\Delta} u_0$ is the unique solution to 
\begin{equation}\label{eq: linheat}
\begin{cases}
			 \partial_t u = \Delta u, \ \mbox{in} \ \ (0,\infty) \times \Omega, \\
			 u_{|\partial\Omega} = 0,\\
		u(0) = u_0 \in H^1_0(\Omega).
		\end{cases}
\end{equation} 
The following $L^p-L^q$ smoothing estimates hold \cite[Proposition 48.4]{QuSo19}.
\begin{prop}\label{prp: lplq}
For any $1 \leq q \leq p \leq \infty$, $t > 0$ and $f \in L^q(\Omega)$, 
 \begin{equation}\label{eq: SmoothEst}
 \left\| e^{t\Delta} f \right\|_{L^p} \leq \frac{1}{t^{\frac{d}{2} \left(\frac 1 p - \frac 1 q\right)}} \| f\|_{L^q}.
 \end{equation}
\end{prop}

As a consequence, we obtain the following space-time estimates.

\begin{prop}\label{prp: sptime}
 Let $2 \leq q,r \leq \infty$ satisfy
 \begin{equation}\label{eq: qrcon}
 \frac 2 q + \frac d r = \frac d 2
 \end{equation}
 with $(q,r,d) \neq (2,\infty,2)$.
 Then for any $f \in L^2$, we have
 \begin{equation}\label{eq: sptime1}
 \left\| e^{t\Delta} f \right\|_{L^q((0,\infty),L^r)} \lesssim \| f\|_{L^2}.
 \end{equation}
 Moreover, if $(\rho,\gamma)$ is another pair satisfying condition \eqref{eq: qrcon}, then 
 \begin{equation}\label{eq: sptime2}
 \left\| \int_0^te^{(t - \tau) \Delta} f(\tau) d\tau \right\|_{L^q((0,\infty),L^r)} \lesssim \| f\|_{L^{\rho'}((0,\infty),L^{\gamma'})},
 \end{equation}
 for any $f \in L^{\rho'}((0,\infty),L^{\gamma'})$.
\end{prop} 

The next proposition presents a gradient smoothing estimate, as outlined in \cite[Theorem 16.3]{LaSoUr88}.

\begin{prop} \label{prop: smooth_par}
There exists a constant $C >0$ such that for any $F \in L^2([0,\infty), L^2(\Omega))$
 \begin{equation} \label{eq: smooth}
 \left\|\nabla \int_0^te^{(t-s)\Delta}F(s)\,ds\right\|_{L^\infty([0,\infty),L^2)}\leq C\|F\|_{L^2([0,\infty),L^2)}.
 \end{equation}
\end{prop}

\subsection{Stationary states.} 

We revisit certain findings regarding the stationary solutions of equation \eqref{eq: flow}, which are solutions to 
\begin{equation}\label{eq:st_st}
	0 = \Delta Q - \omega Q + \mu_Q|Q|^{2\sigma}Q,
\end{equation}
where $\mu_Q \in \R$ is given by 
\begin{equation*}
    \mu_Q = \frac{\| \nabla Q\|_{L^2}^2 + \omega \|  Q\|_{L^2}^2}{\|Q\|_{L^{2\sigma + 2}}^{2\sigma + 2}}.
\end{equation*}
We call ground state a strictly positive solution to \eqref{eq:st_st}. 
We recall \cite[Theorem $1$]{GiNiNi79} and \cite{Kw89}.
	\begin{prop}\label{prp:gs1}
		Let $\Omega = \{ x\in \R^d\, : \, |x| < R\}$ for some $R >0$ or $\Omega = \R^d$. For any $\omega >0$, there exists a unique positive solution $Q_{gs} \in H^1(\Omega)$ to \eqref{eq:st_st}. It is radially symmetric, strictly positive and strictly decreasing in $|x|$.
	\end{prop}
	Now we show that there exists a minimum of the variational problem 
 \begin{equation}
     \label{eqMinF}
     f_m  = \inf \{ F[u]\, : \, u \in V, \ \| u \|_{L^2}^2 = m \}
 \end{equation}
where $m >0$, $F$ is defined in \eqref{eqFIntr} and $V$ is either $H^1_{rad}(\R^d)$ or $H^1(\Omega)$ with $\Omega$ being bounded. 

\begin{prop}\label{lemMinRad}
    Let $\omega >0$ and $V = H^1_{rad}(\R^d)$, $d \geq 2$. There exists $u\in H^1_{rad}(\R^d)$ such that $F[u] = f_m$ defined in \eqref{eqMinF}. $u$ satisfies equation \eqref{eq:st_st}, and it is strictly positive up to a change of sign. In particular, $u = Q_{gs}$ or $u = - Q_{gs}$. 
\end{prop}

\begin{proof}
Let $u_n \in H^1_{rad}$ be a minimizing sequence. Since $| u_n | \geq 0$ satisfies $F[|u_n|] \leq F[u_n]$, we assume that $u_n \geq 0$. We show that $\{u_n\}$ is uniformly bounded in $H^1$. Indeed, if 
$$ \limsup_{n \to \infty} \| \nabla u_n \|_{L^2} \to \infty$$
then by \eqref{eq: GN} 
$$ \limsup_{n \to \infty} F[u_n] \geq \limsup_{n \to \infty}\frac{\| \nabla u_n \|_{L^2}^2 + \omega m}{ C_{GN} m^{2 - \frac{d\sigma}{\sigma + 1}} \| \nabla u_n \|_{L^2}^\frac{d\sigma}{\sigma + 1}} \to \infty$$ 
as $d\sigma/(\sigma + 1) < 2$, which is a contradiction. 

Then, up to a subsequence, we have that $    F[u_n] \leq f_m + \eps_n$ where $\eps_n \to 0 $ as $n \to \infty$. Thus, by \eqref{eq: GN}
\begin{equation*}
    \| \nabla u_n \|_{L^2}^2 + \omega \| u_n \|_{L^2}^2 \lesssim ( f_m + \eps_n)  \| u_n \|_{L^2}^{1 - \frac{d\sigma}{\sigma + 1}}\| \nabla u_n \|_{L^2}^\frac{d\sigma}{\sigma + 1}.
\end{equation*}
Since $\frac{d\sigma}{\sigma + 1} < 2$, we obtain that $\| u_n \|_{H^1} \lesssim 1$ and, up to a subsequence, there exists $0 \leq u \in H^1_{rad}$ such that $u_n \rightharpoonup u$ in $H^1$. Since $H^1_{rad} \hookrightarrow L^{2\sigma + 2}$ is compact, we have 
\begin{equation*}
    F[u] \leq \liminf_{n \to \infty} F[u_n],
\end{equation*}
which yields $\| u_n \|_{H^1} \to \| u \|_{H^1}$. Then $u$ satisfies $F'[u] = \frac{\lambda}{2} u$ for some Lagrange multiplier $\lambda \in \R$, where 
 \begin{equation*}
        F'[u] = \frac{2}{\| u \|_{L^{2\sigma +2}}^2} \left( -\Delta u +  \omega u - \frac{\| \nabla u \|_{L^2}^2 +  \omega \| u \|_{L^2}^2}{\| u \|_{L^{2\sigma + 2}}^{2\sigma + 2}} |u|^{2\sigma}u \right).
    \end{equation*}
 Thus $u$ satisfies
    \begin{equation*}
        - \Delta u +  \omega u - \mu[u] |u|^{2\sigma} u - \frac \lambda 4 \| u\|_{L^{2\sigma + 2}}^2 u = 0.
    \end{equation*}
    By taking the scalar product of this equation with $u$, we obtain that 
    \begin{equation*}
         \frac \lambda 4 \| u\|_{L^{2\sigma + 2}}^2 \|u\|_{L^2}^2 = 0
    \end{equation*}
   and so $\lambda = 0$.
\end{proof}

The hypothesis $\omega >0$ in Proposition \ref{lemMinRad} can be relaxed for bounded domains.

\begin{prop}\label{lemMinBND}
    Let $\omega \in \R$,  $V = H^1(\Omega)$, $\Omega \subset \R^d$ bounded. There exists $u\in V$ such that $F[u] = f_m$ defined in \eqref{eqMinF}.
\end{prop}

\begin{proof}
    The proof is the same as that of Proposition \ref{lemMinRad} and it is based on the compact embeddings $H^1(\Omega) \hookrightarrow L^2(\Omega)$ and  $H^1(\Omega) \hookrightarrow L^{2\sigma + 2} (\Omega)$. For $\omega < 0$, we only need to prove that $F$ is bounded from below. From \eqref{eq: UnifBnPot} we get $F(u) \geq \omega Vol(\Omega)^\frac{\sigma}{\sigma + 1}$. 
\end{proof}

\section{Local Well-posedness}\label{secExist}

\subsection{Local well posedness in $\R^d$.}

We establish the local well-posedness of equation \eqref{eq: flow} using the contraction principle. Due to the unusual nonlinearity, we use separately the space-time estimates \eqref{eq: sptime2} \eqref{eq: smooth}, and consequently the proof is divided into two parts. In the first part, we use \eqref{eq: sptime1} and \eqref{eq: sptime2} to address the interval $\frac{1}{2} \leq \sigma < \frac{2}{(d-2)^+}$. For the case $0 \leq \sigma \leq \frac{1}{2}$, refer to Remark \ref{rmkLocalWP} below. Both results are proven for $\Omega \equiv \R^d$ but can also be applied to bounded domains.

We begin with the first case.

\begin{prop}\label{thm: ConPrin1}
    Let $\frac{1}{2} \leq \sigma < \frac{2}{(d-2)^+}$. Then, for any $u_0\in H^1(\R^d)$, there exists $T_{max}>0$ and a solution $u \in C([0,T_{max}),H^1(\R^d))$ to \eqref{eq: flow}. Moreover, either $T_{max} = \infty$, or $T_{max} < \infty$ and
\begin{equation}\label{eq: bu1}
\lim_{t \rightarrow T_{max}} \| \nabla u(t)\|_{L^2} = \infty.
\end{equation}
\end{prop}

\begin{proof}
    Let  
    \begin{equation*}
        r=2\sigma +2, \quad q =  \frac{4\sigma + 4}{d\sigma},
    \end{equation*}
    and consider the set
    \begin{equation*}
	\begin{aligned}
    		 E = \bigg\{& u \in L^{\infty}\left([0,T], H^1(\R^d) \right) \cap L^{q}\left((0,T), W^{1,r} (\R^d) \right) :  \, \\
		  & \| u\|_{L^{\infty}\left((0,T), H^1\right)} \leq 2M, \ \| u\|_{L^{q}\left((0,T), W^{1,r} \right)} \leq 2M, \ \inf_{t \in [0,T]} \| u \|_{L^{2\sigma + 2}} \geq \frac{N}{2} \bigg\},				 
	\end{aligned}
    \end{equation*}
    where $T,M,N >0$ are to be chosen later. $E$ is a complete metric space with the distance
    \begin{equation*}
    	\mathrm{d}(u, v)=\|u-v\|_{L^{q}\left((0, T), W^{1,r}\right)}+\|u-v\|_{L^{\infty}\left([0, T], H^1\right)}.
    \end{equation*}
    Fix $u,v \in E$. Since $H^{1}\left(\R^{d}\right) \hookrightarrow L^{r}\left( \R^d \right)$, we obtain
\begin{equation*}
	\begin{aligned}
		\left\||u|^{2\sigma} u - |v|^{2\sigma} v \right\|_{L^{q^{\prime}}\left((0, T), L^{r^\prime}\right)} \lesssim 
		\left(\|u\|_{L^{\infty}\left([0, T], L^{r}\right)}^{2\sigma}+\|v\|_{L^{\infty}\left([0,T], L^{r}\right)}^{2\sigma}\right)\|u-v\|_{L^{q^\prime}\left((0, T), L^{r}\right)},
	\end{aligned}
\end{equation*}
\begin{equation*}
	\left\| \nabla \left( |u|^{2\sigma} u - |v|^{2\sigma} v  \right) \right\|_{L^{r^{\prime}}}  \lesssim \left( \|u\|_{L^{r}}^{2\sigma-1} + \|v\|_{L^{r}}^{2\sigma-1}\right) \left( \|\nabla u\|_{L^{r}} + \|\nabla v\|_{L^{r}} \right)  \|u - v\|_{H^1},
\end{equation*}
and
\begin{equation}\label{eq: omega2}
	\begin{aligned}
		&\left\|\nabla \left(|u|^{2\sigma} u - |v|^{2\sigma} v \right) \right\|_{L^{q^{\prime}}\left((0, T), L^{r^{\prime}}\right)} \lesssim \left(\|u\|_{L^{\infty}\left([0,T], L^{r}\right)}^{2\sigma-1}+\|v\|_{L^{\infty}\left([0,T], L^{r}\right)}^{2\sigma-1}\right) \\ 
		&\left( \| \nabla u\|_{L^{q^\prime}((0, T), L^{r})} + \| \nabla v\|_{L^{q^\prime}((0, T), L^{r})} \right)\|u-v\|_{L^{\infty}\left([0,T], H^1\right)}.
	\end{aligned} 
\end{equation}
Using H\"older's inequality, we deduce 
\begin{equation*}
	 \left\||u|^{2\sigma}u \right\|_{L^{q^{\prime}}\left((0, T), W^{1, r^{\prime}}\right)} \lesssim \left(T+T^{\frac{q-q^{\prime}}{q q^{\prime}}}\right)\left(1+M^{2\sigma}\right) M
\end{equation*}
 and
\begin{equation*}
	\left\||u|^{2\sigma}u - |v|^{2\sigma} v \right\|_{L^{q^{\prime}}\left((0, T), W^{1, r^{\prime}}\right)}\lesssim \\
	\left(T+T^{\frac{q-q^{\prime}}{q q^{\prime}}}\right)\left(1+M^{2\sigma}\right) \mathrm{d}(u, v) .
\end{equation*}
We notice 
\begin{equation*}
	\begin{aligned}
	\left\|  \mu[u]u - \mu[v] v \right\|_{L^{1}\left((0, T),H^1 \right)}  &\lesssim T\left( \| \mu[u]\|_{L^\infty[0,T]} \| u -  v\|_{L^{\infty}\left([0,T],H^1 \right)}  \right.\\ 
	 & \left. +\|  u\|_{L^{\infty}\left([0,T],H^1 \right)} \left\|  \mu[u] - \mu[v] \right \|_{L^{\infty}[0,T]}  \right),
	\end{aligned}
\end{equation*}
\begin{equation}\label{eq: muHeatEst}
 \| \mu[u]\|_{L^\infty[0,T]} \leq\frac{\max(1, \omega) \| u \|_{L^\infty([0,T],H^1)}^2}{N^{2\sigma + 2}} \lesssim 1
\end{equation}
and 
\begin{equation}\label{eq: omega1}
\begin{aligned}
 \left\| \mu[u] - \mu[v] \right\|_{L^{\infty}[0,T]} \leq C \| u- v \|_{L^{\infty}\left([0,T],H^1 \right)}.
\end{aligned}
\end{equation}
We denote by $\mathcal{F}(u)$ the map
\begin{equation}\label{eq: IntegForm}
	\mathcal{F}(u)(t)=e^{t\Delta} u_0+ \int_{0}^{t} e^{(t-s)\Delta} \left( \mu[u]|u|^{2\sigma} u - \omega u \right) \, ds .
\end{equation}
We obtain 
\begin{equation}\label{eq: strHeat1}
	\begin{split}
	\| \mathcal F (u)\|_{L^q([0,T),L^r) \cap L^\infty([0,T],L^2)} &\lesssim \| u_0\|_{L^2} +  T^{\frac{q-q^{\prime}}{q q^{\prime}}}  \| \mu[u] \|_{L^\infty[0,T]} \| u\|^{2\sigma}_{L^\infty([0,T],H^1)} \| u\|_{L^q([0,T),L^r)} \\ 
	&+ T \| u\|_{L^\infty([0,T],L^2)}
	\end{split}
\end{equation}
and 
\begin{equation}\label{eq: strHeat2}
	\begin{aligned}
		\| \nabla \mathcal F (u) \|_{L^q([0,T),L^r) \cap L^\infty([0,T],L^2)} & \lesssim \| \nabla u_0\|_{L^2} + T^{\frac{q-q^{\prime}}{q q^{\prime}}}   \| \mu[u] \|_{L^\infty[0,T]} \| u\|^{2\sigma}_{L^\infty([0,T],H^1)} \| \nabla u\|_{L^q([0,T),L^r)} \\
		&+ T \|\nabla u\|_{L^\infty([0,T],L^2)}.
	\end{aligned}
\end{equation}
As a consequence of \eqref{eq: strHeat1}, \eqref{eq: strHeat2} and \eqref{eq: muHeatEst}, there exists $K = K(M,N) >0$ such that
\begin{equation}\label{eq: upbond1}
\|\mathcal F (u)\|_{L^\infty([0,T], H^1) \cap L^{q}\left((0, T), W^{1, r}\right)} \lesssim  \|u_0\|_{H^{1}}+ \left(T+T^{\frac{q-q^{\prime}}{q q}}\right)K.
\end{equation}
Similarly, from \eqref{eq: omega1} we get 
\begin{equation}\label{eq: upbond2}
		\|\mathcal F (u)-\mathcal F (v)\|_{L^{\infty}\left(\left[0,T], H^1 \right)\right.} +\|\mathcal F (u)-\mathcal F (v)\|_{L^{q}\left((0, T), W^{1,r}\right)} \lesssim \left(T+T^{\frac{q-q^{\prime}}{q q^{\prime}}}\right)K \mathrm{d}(u, v) .
\end{equation}
Finally, we observe that
\begin{equation*}
 \begin{aligned}
 \big\| \mathcal F (u) (t) \big\|_{L^{2\sigma +2}} & \geq \| e^{t\Delta} u_0 \|_{L^{2\sigma +2}} - \bigg\| \int_{0}^{t} e^{(t-s)\Delta} \left( \mu[u]|u|^{2\sigma} u - \omega u \right) \, ds \bigg\|_{L^{2\sigma +2} }.
 \end{aligned}
\end{equation*}
Using the smoothing estimates \eqref{eq: SmoothEst} we obtain
\begin{equation*}
\begin{aligned}
   \bigg\| \int_{0}^{t} e^{(t-s)\Delta} \left( \mu[u]|u|^{2\sigma} u - \omega u \right) \, ds \bigg\|_{L^{2\sigma +2} } &\leq \int_0^t \mu[u(s)] \frac{1}{4\pi(t-s)^\frac{d\sigma}{2\sigma + 2}} \| u(s)\|_{L^{2\sigma + 2}} \,ds \\
   & + \omega \int_0^t \| u (s) \|_{L^{2\sigma + 2}}\, ds \\
   \\ &\lesssim (\| \mu \|_{L^\infty[0,t]} t^{1 - \frac{d\sigma}{2\sigma + 2}} + t) \| u \|_{L^\infty([0,t], L^{2\sigma + 2})}.
   \end{aligned}
\end{equation*}
Since $\frac{d\sigma}{2\sigma + 2} < 1$, we can use \eqref{eq: muHeatEst} to conclude that there exists 
 $K_1(M,N,T) = K_1(T) >0 $ with $K_1(T) \to 0$ as $T \to 0$  such that
\begin{equation*}
 \inf_{t \in [0,T]} \big\| \mathcal F (u) (t) \big\|_{L^{2\sigma + 2}} \geq  \inf_{t \in [0,T]}\| e^{t\Delta} u_0\|_{L^{2\sigma + 2}} - K_1 (T).
\end{equation*}
Since $e^{t\Delta} u_0 \in C([0,T], L^{2\sigma + 2})$ (see \cite[Theorem $15.2$]{QuSo19}), there exists $T >0$ small enough, such that 
\begin{equation}\label{eq: lwBond}
         \inf_{t \in [0,T]} \| F(u)(t) \|_{L^{2\sigma +2}}^2 \geq   \frac{1}{2} \|  u_0\|_{L^{2\sigma + 2}}
    \end{equation}
We set $N =  \| u_0 \|_{L^{2\sigma + 2}}$ and $M = \| u_0 \|_{H^1}$. For $T >0$ small enough,  \eqref{eq: lwBond}, \eqref{eq: upbond1} and \eqref{eq: upbond1} imply that  $\mathcal F$ maps $E$ into itself and is a contraction and the fixed point is a solution to \eqref{eq: flow}.
\end{proof}
\begin{remark}\label{rmkLocalWP}
   In the proof of Proposition \ref{thm: ConPrin1}, we use the distance $\mathrm{d}(u, v)=\|u-v\|_{L^{q}\left((0, T), W^{1,r}\right)}+\|u-v\|_{L^{\infty}\left([0, T], H^1\right)}$. This metric is required to effectively manage the nonlocal nonlinearity $\mu[u]$. However,  the power-type nonlinearity is locally Lipschitz continuous in $H^1$ only for $\frac{1}{2} \leq \sigma$.  In the other hand, due to the presence of the $H^1$ norm in the definition of $\mu[u]$,  it is not possible to use the weaker metric $\mathrm{g}(u, v)=\|u-v\|_{L^{q}\left((0, T), L^r \right)} + \|u-v\|_{L^{\infty}\left((0, T), L^2 \right)}$ as typically done (see for instance \cite[Theorem 4.4.1.]{Ca03}). 
\end{remark}

We now show the proof for the case $0<\sigma < \frac{1}{(d-2)^+}$. It is based on the contraction principle as before, and thus the details are omitted, the main idea being that, for $0 < \sigma < \ \frac{1}{d-2}$, the embedding $H^1 \hookrightarrow L^{4\sigma + 2}$ holds

\begin{prop}\label{thm: proof1}
		Let $0 \leq \sigma < \frac{1}{(d-2)^+}$. Then for any $u_0 \in H^1_0(\R^d)$, there exists a maximal time of existence $T_{max}>0$ and a unique solution $u \in C([0,T_{max}), H^1(\R^d))$ to \eqref{eq: flow}. Moreover, either $T_{max} = \infty$, or $T_{max} < \infty$ and 
		\begin{equation*}
			\lim_{t \rightarrow T_{max}} \| \nabla u(t)\|_{L^2} = \infty.
		\end{equation*}
\end{prop}

\begin{proof}
	We define the set
	\begin{equation*}
		\mathcal{X} = \{ u \in L^\infty([0,T], H^1(\R^d)),\, \| u\|_{L^\infty([0,T],H^1)} \leq 2M, \, \inf_{[0,T]} \| u\|_{L^2} \geq \frac{N}{2} \},
	\end{equation*}
	and we  choose later $T, M, N >0$. This space is equipped with the distance 
	\begin{equation*}
		d(u,v) = \| u - v\|_{L^\infty([0,T],H^1)}. 
	\end{equation*}
	Given any $u_0 \in H^1(\R^d)$, $u_0 \centernot{\equiv} 0$, $u,v\in \mathcal X$, we consider the map $\mathcal{F}(u)(t)$ defined in \eqref{eq: IntegForm}. From \eqref{eq: strHeat1} and \eqref{eq: lwBond} we see that there exists $K(u_0,M,N,T) = K(T) >0 $ such that $K(T) \to 0$ as $T \to 0$ so that
	\begin{equation}\label{eq: strHeat11}
		\| \mathcal F (u)\|_{L^\infty([0,T],L^2)} \lesssim \| u_0\|_{L^2} +  K(T)
	\end{equation}
    and 
    \begin{equation*}
        \inf_{t \in [0,T]} \big\| \mathcal F (u) (t) \big\|_{L^{2\sigma + 2}} \geq  \inf_{t \in [0,T]}\| e^{t\Delta} u_0\|_{L^{2\sigma + 2}} - K (T).
    \end{equation*}
	Furthermore, by \eqref{eq: smooth} we infer that
		\begin{equation*}
        \begin{aligned}
            \|	\nabla \mathcal{F}(u)\|_{L^\infty([0,T],L^2)} &\lesssim \| u_0\|_{H^1} + T^\frac{1}{2} \| \mu[u]|u|^{2\sigma} u - \omega  u \|_{L^\infty([0,T], L^2)} \\ 
            & \lesssim  \| u_0\|_{H^1} + T^\frac{1}{2} \left(  \| \mu[u] \|_{L^{\infty}} \| u\|^{2\sigma + 1}_{L^{\infty}([0,T]L^{4\sigma +2})} + \| u \|_{L^\infty([0,T], L^2)}\right).
        \end{aligned}
		\end{equation*}
	Notice that $\sigma < \frac{1}{(d-2)^+}$ implies $4\sigma +2 < \frac{2d}{(d-2)^+},$ and so there exists $K_1(M,N,T) = K_1(T) >0$ such that $K_1(T) \to 0$ as $T \to 0$ and   
	\begin{equation}\label{eq: strHeat14}
			\|\nabla\mathcal{F}(u)\|_{L^\infty([0,T],L^2)} \leq \|u_0\|_{H^1} + K_1 (T).
	\end{equation}
    A standard choice of $M, N$, and $T$ can be done to have that $\mathcal F$ maps $\mathcal{X}$ in itself. Similar computations lead to see that it is a contraction. 
\end{proof}

\begin{remark}
    Theorems \ref{thm: ConPrin1} and \ref{thm: proof1} cover the whole subcritical interval $0 < \sigma < \frac{2}{d - 2}$ for $d \leq 4$. 
\end{remark}

\subsection{Schauder fixed point theorem}
Here we show the local well-posedness employing the Schauder fixed point theorem. This proof covers the entire energy subcritical regime in any dimension in a bounded domain. The proof aims to avoid using the contraction argument involving the product of $\mu[u]$ and $|u|^{2\sigma}u$ in \eqref{eq: flow}.  

\begin{lem}\label{lem: fixed}
		Let $0<\sigma < \frac{2}{(d-2)^+}$ and $\lambda \in L^\infty(\R)$. Then there exists $0<T=T(\|u_0\|_{H^1_0}, \|\lambda\|_{L^\infty})$ such that the Cauchy problem 
		\begin{equation}\label{eq: heat_mu}
			\begin{cases}
			\partial_t u = \Delta u - \omega u  + \lambda |u|^{2\sigma} u , \\
			u_{|\partial\Omega} = 0,\\
			u(0) = u_0 \in H^1_0(\Omega),
			\end{cases}
		\end{equation}	
		admits a unique solution $u \in C([0,T], H^1_0(\Omega))$.    Finally, if $u_0 \in H^2(\Omega) \cap H^1_0(\Omega)$ then  $u \in C([0,T], H^2(\Omega) \cap H^1_0(\Omega))$.
\end{lem}

The proof of this lemma is mostly identical to that of Proposition \ref{thm: ConPrin1}. It relays on the contraction principle with respect to the metric $\mathrm{g}(u, v)=\|u-v\|_{L^{q}\left((0, T), L^r\right)}  + \|u-v\|_{L^{\infty}\left((0, T), L^2 \right)}$. The proof is standard, see for example \cite{QuSo19}. The next lemma provides a priori bounds on the approximating sequence. These allow us to find the solution to \eqref{eq: flow} as a limit of this sequence in some suitable spaces. 

\begin{lem}\label{prp: it_time}
		Let $u_0 \in  H^1$ and $u \in C([0,T], H^1_0(\Omega))$ be the  corresponding solution to \eqref{eq: heat_mu}. Then there exists $0<T^*=T^*(\|u_0\|_{H^1},\| \lambda\|_{L^\infty}) \leq T$, so that
		\begin{equation}\label{eq: con}
			\sup_{t \in [0,T^*]} \| u(t)\|_{H^1} \leq 2 \| u_0\|_{H^1}, 
		\end{equation} 
		and
		\begin{equation}\label{eq: con2}
			\inf_{t \in [0,T^*]} \| u(t) \|_{L^{2\sigma +2}} \geq \frac{1}{2} \| u_0\|_{L^{2\sigma +2}}.
		\end{equation}
		Moreover, there exists $\mathcal{C}(\| u_0\|_{H^1}) = \mathcal{C} >0$ such that
		\begin{equation}\label{eq: mub}
			\sup_{[0,T^*]} \mu[u(t)] \leq \mathcal{C}.
		\end{equation} 
  Finally, if $u_0 \in H^2(\Omega) \cap H^1_0(\Omega) $ then
    \begin{equation}\label{eqCont2}
        \sup_{t \in [0,T^*]} \| u(t)\|_{H^2} \leq 2 \| u_0\|_{H^2}. 
    \end{equation}
\end{lem}
	
\begin{proof}
	By continuity, since $u \in C([0,T], H^1_0(\Omega))$, we can find $T_1(\| u_0\|_{H^1},\| \lambda \|_{L^\infty[0,T_1]}) >0$ such that \eqref{eq: con} and \eqref{eq: con2} are true. \eqref{eq: mub} is a direct consequence. Finally, for $H^2$ data, we can use the persistence of regularity and continuity with respect to initial conditions. 
\end{proof}
 
Now, we employ the Schauder fixed point theorem to find a solution to \eqref{eq: flow} for any data in $H^2(\Omega) \cap H^1_0(\Omega)$. This requires an intermediate step where the nonlocal term is modified to 
\begin{equation}
    \label{eqMuEps}
    \mu_\eps[u] = \frac{\| \nabla u \|_{L^2}^2 + \omega \| u \|_{L^2}^2}{\| u \|_{L^{2\sigma + 2}}^{2\sigma + 2} + \eps}
\end{equation}
for $\eps >0$. This modification is introduced to ensure that the nonlocal term is bounded from above in a ball in  $L^\infty([0, T], H^1_0(\Omega))$ for some $T >0$.  We will later show that it is possible to take the limit $\eps \to 0$ to obtain a solution to \eqref{eq: flow}. The idea is that, as $\eps \to 0$, the $L^2$ norm of solutions to \eqref{eqHeatMuA} below admits a uniform lower bound in a small time interval. Due to \eqref{eq: UnifBnPot}, this yields a uniform lower bound also for the $L^{2\sigma +2}$ norm. Finally, by using a density argument, we will demonstrate that solutions to \eqref{eq: flow} exist for $H^1_0$ data. 

\begin{prop}
    Let $\Omega \subset \R^d$ be a bounded domain with $C^2$ boundary. For any $ u_0 \in H^2(\Omega) \cap H^1_0(\Omega)$, there exists a solution $u\in C([0,T),H^2(\Omega) \cap H^1_0(\Omega))$ to \eqref{eq: flow}.
\end{prop}
\begin{proof}
Let $F_{u_0,\eps}: C([0,T_{u_0,\eps}], H^1_0(\Omega)) \to C([0,T_{u_0,\eps}], H^1_0(\Omega))$ be such that $F_{u_0,\eps}(u) = v$ where $v$ satisfies
\begin{equation*}
    \begin{cases}
        \partial_t v = \Delta v  - \omega v + \mu_\eps[u]|v|^{2\sigma} v, \\
        v(0) = u_0.
    \end{cases}
\end{equation*}
  Let 
\begin{equation*}
    B_{u_0} = \left\{ u \in C([0,T_{u_0,\eps}], H^1_0(\Omega))\, : \, \| u \|_{L^\infty([0,T_{u_0,\eps}],H^1)} \leq 2 \| u_0 \|_{H^1} \right\}
\end{equation*}
Let $F_{u_0,\eps}(B_{u_0}) =  I_{u_0}$.  Notice that, for any $u \in B_{u_0}$, $F_{u_0,\eps}(u) \in C([0,T_{u_0,\eps}], H^2(\Omega) \cap H^1_0(\Omega) )$ and, by \eqref{eqCont2}, there exists $T_{u_0,\eps} >0 $ such that $I_{u_0} \subset B_{u_0}$ and  
\begin{equation*}
    \| F_{u_0,\eps}(u) \|_{L^\infty([0,\tilde T], H^2)} \leq 2 \| u_0\|_{H^2}.
\end{equation*} 
We show that $I_{u_0}$ is precompact. This follows from the fact that
    \begin{equation*}
        I_{u_0} \subset W = \left\{ v \in C([0,T_{u_0,\eps}], H^2(\Omega) \cap H^1_0(\Omega) )\,:\, \partial_t v \in L^2([0,T_{u_0,\eps}], L^2(\Omega)) \right\}.
    \end{equation*} 
    and the Aubin-Lions lemma  \cite[Theorem II.$5.16$]{BoFa12} which implies that the embedding of $W$ in $C([0,T_{u_0,\eps}], H^1_0(\Omega))$ is compact.
     Thus, applying the Schauder fixed point theorem \cite{GiTr98}, we find that $F_{u_0,\eps}$ admits a fixed point $u \in C([0,T_{u_0,\eps}], H^1_0(\Omega))$ satisfying 
    \begin{equation}\label{eqHeatMuA}
        \begin{cases}
        \partial_t u_\eps = \Delta u_\eps  - \omega u_\eps + \mu_{\eps}[u_\eps]|u_\eps|^{2\sigma} u_\eps, \\
        u_\eps(0) = u_0 \in H^2(\Omega) \cap H^1_0(\Omega) .
        \end{cases}
    \end{equation}
    By the persistence of regularity, we conclude that $u\in C([0,T_{u_0,\eps}], H^2(\Omega) \cap H^1_0(\Omega) )$.  
    By taking the scalar product of \eqref{eqHeatMuA} with $u_\eps$ we get
    \begin{equation}
        \label{eqMassUeps}
        \| u_\eps(t) \|_{L^2}^2 = \| u_0 \|_{L^2}^2 - 2\eps \int_0^t \mu_\eps[u_\eps(s)] ds.
    \end{equation}
    Since $\| u_\eps(t) \|_{L^2}^{2\sigma + 2} \leq Vol(\Omega)^{\sigma} \| u_\eps(t) \|_{L^{2\sigma +2}}^{2\sigma + 2}$ 
    we obtain 
    \begin{equation}
        \label{eqMuEps1}
        \mu_\eps[u_\eps(t)] \lesssim \frac{\| u_0 \|_{H^1}^2}{\left(\| u_0 \|_{L^2}^2 - 2\eps \int_0^t \mu_\eps[u_\eps(s)] ds\right)^{\sigma + 1} + \eps}.
    \end{equation}
    From
    \begin{equation*}
         \| \mu_\eps[u_\eps(s)] \|_{L^\infty[0,t]} \lesssim \frac{\| u_0 \|_{H^1}^2}{\eps},
    \end{equation*}
    and \eqref{eqMuEps} we get 
    \begin{equation*}
        \mu_\eps[u_\eps(t)] \lesssim \frac{\| u_0 \|_{H^1}^2}{\left(\| u_0 \|_{L^2}^2 - C t \| u_0 \|_{H^1}^2\right)^{\sigma + 1} + \eps}.
    \end{equation*}
    for some $C >0$. Thus there exists $\eps_0 (u_0) > 0$ and $T_0 = T(u_0) > 0$ such that for any $0 < \eps \leq \eps_0$, $\| \mu_\eps[u_\eps(t)] \|_{L^\infty[0,T_0]} \leq K$, where $K$ does not depend on $\eps$. In particular, for $\eps < \eps_0$, $u_\eps \in C([0,T_{0}], H^2(\Omega) \cap H^1_0(\Omega) )$ with $\partial_t u_\eps \in L^2([0,T_{0}], L^2(\Omega))$, $\| u_\eps \|_{H^1} \leq 2 \| u_0 \|_{H^1}$ and $\| \mu_\eps[u_\eps] \|_{L^\infty} \lesssim 1$. 
    
    Finally, let $\{\eps_n\}$ such that $\eps_n \to 0$ as $n\to \infty$ and $0 < \eps_n < \eps_0$. Then, there exists a subsequence, still denoted by $u_{\eps_n}$ such that 
	\begin{equation*}
		u_{\eps_n} \overset{\ast}{\rightharpoonup} u \ \mbox{ in } \ L^{\infty}([0,T], H^1_0(\Omega)) \quad \mbox{ and } \quad 	u_{\eps_n} \rightharpoonup u \ \mbox{ in } \ L^{2}([0,T], H^2(\Omega)). 
	\end{equation*}
	The Aubin-Lions lemma implies that $u_{\eps_n} \rightarrow u$ strongly in $L^2([0,T_0], H^1_0(\Omega))$ and in $C([0,T_0], L^2(\Omega))$. Therefore, $\mu_{\eps_n}[u_{\eps_n}] \rightarrow \mu[u]$ strongly in $L^2([0,T_0])$. Thus, $u$ satisfies \eqref{eq: flow}. 
\end{proof}

We show that \eqref{eq: flow} admits solutions for $H^1$-data. 
\begin{prop}\label{prpLocEx}
    Let $\Omega \subset \R^d$ be a bounded domain with $C^2$ boundary. For any $u_0 \in  H^1_0(\Omega)$, there exists a solution $u\in C([0,T), H^1_0(\Omega))$ to \eqref{eq: flow}.
\end{prop}

    \begin{proof}
    Let $u_0 \in H^1_0(\Omega) $ and let $\{u_0^{(n)}\}_{n \in \N} \subset H^2(\Omega) \cap H^1_0(\Omega) $ be such that  $ \| u_0^{(n)} \|_{H^1} \lesssim 1$ and
    \begin{equation*}
      \lim_{n \to \infty}  \| u_0^{(n)} - u_0 \|_{H^1} = 0.
    \end{equation*}
    Let $\{u^{(n)}\}_{n \in N} \subset C([0,T_n],H^2(\Omega) \cap H^1_0(\Omega) )$ be the set solutions to \eqref{eq: flow} emanating from $u_0^{(n)}$. By Lemma \ref{prp: it_time}, there exists $T = T(u_0) >0$, such that
    $$\{u^{(n)}\}_{n \in N} \subset C([0,T],H^2(\Omega) \cap H^1_0(\Omega) )$$ 
    and 
    \begin{equation*}
        \sup_{n\in \N, t\in [0,T]} \| u^{(n)} \|_{H^1} \lesssim 1, \quad \sup_{n\in \N} \| \mu[u^{(n)}] \|_{L^\infty[0,T]} \lesssim 1.
    \end{equation*}
     Moreover, $\{u^{(n)}\}$ is uniformly bounded in $L^2([0, T],H^2(\Omega))$. Indeed, we multiply equation \eqref{eqHeatMuA} by $\Delta u^{(n)}$ and integrate in space to get
	\begin{equation*}
	 \begin{aligned}
	 -\frac{1}{2} \frac{d}{dt} \| \nabla u^{(n)}\|_{L^2}^2 &= \| \Delta u^{(n)}\|_{L^2}^2 \\& - (2\sigma + 1) \mu[u^{(n)}] \int_\Omega |u^{(n)}|^{2\sigma} |\nabla u^{(n)}|^2 dx - \omega \| \nabla u^{(n)}\|_{L^2}^2 .
	 \end{aligned}
	\end{equation*}
	Integrating in time leads to
	\begin{equation*}
 \begin{aligned}
 \int_0^{T} \| \Delta u^{(n)}\|_{L^2}^2 \, ds &\lesssim 1 + T \| \nabla u^{(n)}\|_{L^{\infty}([0,T],L^2)}\\
 &+  \|\mu[u^{(n)}]  \int_0^{T} \int_\Omega |u^{(n)}|^{2\sigma} |\nabla u^{(n)}|^2 \,dx \,ds.
 \end{aligned}
	\end{equation*}
	We suppose that $d \geq 3$, the case $d \leq 2$ being easier. We have
	\begin{equation*}
		\int_\Omega |u^{(n)}|^{2\sigma} |\nabla u^{(n)}|^2 dx \leq \| u^{(n)}\|_{L^\frac{2d}{d-2}}^{2\sigma} \| \nabla u^{(n)} \|_{L^\frac{2d}{d-\sigma(d-2)}}^2,
	\end{equation*}
	and thus
 \begin{equation*}
 \int_0^{T} \int_\Omega |u^{(n)}|^{2\sigma} |\nabla u^{(n)}|^2 \, dx \, ds \leq \|u^{(n)} \|_{L^\infty([0,T],H^1)}^{2\sigma} (T)^{1/p} \| \nabla u^{(n)}\|_{L^\frac{4}{\sigma(d-2)}([0,T],L^\frac{2d}{d-\sigma(d-2)})}^2,
	\end{equation*}
 where 
 \begin{equation*}
 0 < p = \frac{2}{2-\sigma(d-2)} < \infty.
 \end{equation*}
	Consequently, we obtain 
    $$\{u^{(n)}\}_{n\in \N} \subset L^{\infty}([0,T], H^1_0(\Omega)) \cap L^2([0,T], H^2(\Omega) \cap H^1_0(\Omega)).$$ From the embedding $H^2 \hookrightarrow L^{4\sigma + 2}$, we also have $\{ \partial_t u^{(n)}\}_{n\in \N}  \subset L^2([0,T], L^2(\Omega))$ and the two sequences are uniformly bounded in these spaces.
 	
So there exists a subsequence $\{u^{(n_k)}\}_{k \in \N}$, and $u \in L^{\infty}([0,T], H^1_0(\Omega)) \cap L^{2}([0,T], H^2(\Omega) \cap H^1(\Omega))$ with $\partial_t u \in L^{2}([0,T], L^2(\Omega))$, such that
	\begin{equation*}
		u^{(n_k)} \overset{\ast}{\rightharpoonup} u \ \mbox{ in } \ L^{\infty}([0,T], H^1_0(\Omega)) \quad \mbox{ and } \quad 	u^{(n_k)} \rightharpoonup u \ \mbox{ in } \ L^{2}([0,T], H^2(\Omega)). 
	\end{equation*}
	The Aubin-Lions lemma implies that $u^{(n_k)} \rightarrow u$ strongly in $L^2([0,T], H^1_0(\Omega))$ and in $C([0,T], L^2(\Omega))$. Therefore, $\mu_\omega[u^{(n_k)}] \rightarrow \mu_\omega[u]$ strongly in $L^2([0,T])$. Thus, $u$ satisfies \eqref{eqHeatMuA}. 
 \end{proof}

\begin{prop}\label{prp: uni}
			 For any $u_0\in H^1_0(\Omega)$ there exists a unique strong solution to \eqref{eq: flow}.
\end{prop}

\begin{proof}
Let $u,v \in C([0,T],H^1_0(\Omega))$ be two solutions to \eqref{eq: flow}. Here $T > 0$ is such that $T < \min (T_u,T_v)$ where $T_u, T_v > 0$ are the maximal times of existence of $u,v$. In particular, there exists $C = C(\|u_0\|_{H^1}) >0$ such that 
			\begin{equation}\label{eq: UpBondUniq}
				\max \left(\|u\|_{L^\infty([0,T],H^1)},  \|v\|_{L^\infty([0,T],H^1)}, \sup_{t\in[0,T]}\mu[u], \sup_{t\in[0,T]} \mu[v] \right) \leq C.
			\end{equation}
By multiplying the equation 
\begin{equation*}
 \partial_t (u -v) = \Delta(u-v) + \mu[u] |u|^{2\sigma} u - \mu[v] |v|^{2\sigma} v -\omega( u - v)
\end{equation*}
by $(u-v)$ and integrating in space, we obtain
\begin{equation}\label{eq: took1}
\begin{split}
 \frac{1}{2} \frac{d}{dt} \| u -v \|_{L^2}^2 &= -\| \nabla (u - v)\|_{L^2}^2 - \omega \| u - v\|_{L^2}^2  
 + \int (\mu[u] |u|^{2\sigma}u - \mu[v]|v|^{2\sigma }v)(u-v) \,dx.
 \end{split}
\end{equation}
We have
\begin{equation*}
    \begin{aligned}
       \int (\mu[u] |u|^{2\sigma}u - \mu[v]|v|^{2\sigma }v)(u-v) \,dx  &=  \mu[u] \int (|u|^{2\sigma}u - |v|^{2\sigma }v)(u-v) \,dx \\
       &+ (\mu[u] - \mu[v]) \int |v|^{2\sigma} v(u -v) \, dx.
    \end{aligned}
\end{equation*}
By Young's inequality, we see that 
\begin{equation*}
\begin{aligned}
       \int |v|^{2\sigma} v(u -v) \, dx &\leq \| v\|_{L^{2\sigma + 2}}^{2\sigma + 1} \| u -v \|_{L^{2\sigma + 2}} \\ 
       &\lesssim \| v\|_{L^{2\sigma + 2}}^{2\sigma + 1}\left(  \| u -v \|_{L^2}^{1 - \frac{d\sigma}{2\sigma + 2}}  \| \nabla( u -v ) \|_{L^2}^ \frac{d\sigma}{2\sigma + 2}  \right) \\
       &\lesssim \| v\|_{L^{2\sigma + 2}}^{2\sigma + 1} \left( \frac{1}{\delta}  \| u -v \|_{L^2}^2 + \delta \| \nabla( u -v) \|_{L^2}^2    \right).
\end{aligned}
\end{equation*}
where $\delta >0$ is to be chosen. So, after a rearrangement, from \eqref{eq: took1} we obtain that there exists $K > 0$ such that
\begin{equation*}
\begin{aligned}
     \frac{1}{2} \frac{d}{dt} \| u - v\|_{L^2}^2 &\leq K ( \mu[u] + \mu[v]) (( \| u \|_{L^{2\sigma + 2}}^{2\sigma + 1} + \| v \|_{L^{2\sigma + 2}}^{2\sigma + 1}) \delta -1) \| \nabla (u - v) \|_{L^2}^2 \\
     & + ( K ( \mu[u] + \mu[v]) ( \| u \|_{L^{2\sigma + 2}}^{2\sigma + 1} + \| v \|_{L^{2\sigma + 2}}^{2\sigma + 1}) \frac{1}{\delta} -\omega) \| u - v \|_{L^2}^2
\end{aligned}
\end{equation*}
We choose $\delta$ such that $ ( \| u \|_{L^{2\sigma + 2}}^{2\sigma + 1} + \| v \|_{L^{2\sigma + 2}}^{2\sigma + 1}) \delta -1 = 0$ and obtain 
\begin{equation*}
     \frac{1}{2} \frac{d}{dt} \| u - v\|_{L^2}^2 \leq \mathcal{C} \| u - v\|_{L^2}^2
\end{equation*}
where $\mathcal{C} = \mathcal{C}(\omega, \delta, \| u_0\|_{H^1}) < \infty$. Since $u(0) = v(0) = u_0$, by Gronwall's Lemma, we conclude that $v(t) = u(t)$ for any $t \in [0,T]$.
\end{proof} 

\begin{remark}
The approach developed here fails for $\Omega=\R^d$. Indeed, the Aubin-Lions lemma would only provide the convergence
\begin{equation*}
u^{(k)}\to u,\quad\textrm{in}\;L^2([0, \tilde T],H^1_{loc}(\R^d)),
\end{equation*}
which is not enough to define $\mu[u]$, as the denominator could converge to zero.
\end{remark}
We are now ready to prove Theorem \ref{thm: existence}.
\begin{proof}[Proof of Theorem \ref{thm: existence}]
    The proof is derived from Propositions \ref{thm: ConPrin1} \ref{thm: proof1}, \ref{prpLocEx} and \ref{prp: uni}
\end{proof}
\section{Global Well-posedness}

Let $F: H^1_0(\Omega) \setminus \{0\} \to \R$ be defined as 
\begin{equation}\label{eq: lyapF}
    F[u(t)] =  \frac{\|\nabla u(t) \|_{L^2}^2 + \omega \| u(t) \|_{L^2}^2}{ \| u(t) \|_{L^{2\sigma + 2}}^2}.
\end{equation}

The next lemma demonstrates that $F$ serves as a Lyapunov functional for the dynamics of \eqref{eq: flow}. 
\begin{lem}
    Let $u_0 \in H^1_0(\Omega)$ and let $u(t)\in C\big([0,T_{max}), H^1_0(\Omega)\big)$ be the corresponding solution to \eqref{eq: flow}. Then, for any $t\in [0,T_{max})$ we have 
    \begin{equation}\label{eq: massConserv}
        \| u(t) \|_{L^2} = \| u_0 \|_{L^2}.
    \end{equation}
    Moreover, if $\| \nabla u(t_0) \|_{L^2}^2 + \omega \| u_0 \|_{L^2}^2 = 0$ for some $t_0 >0$ then there exists $\eps >0$ such that 
         \begin{equation}\label{eq: KinDer}
             \| \nabla u(t_0 + \eps) \|_{L^2}^2 + \omega \| u_0 \|_{L^2}^2 < 0
         \end{equation}
         and otherwise for any $t > t_0$, 
    \begin{equation}\label{eq: estFundamental}
        F[u(t)] =  F[u(t_0)] \, \exp \left(-2 \int_{t_0}^t \frac{\| \partial_s u(s)\|_{L^2}^2}{\|\nabla u(s) \|_{L^2}^2 + \omega \| u(s) \|_{L^2}^2} \, ds \right).
    \end{equation}
\end{lem}

\begin{proof}
    First, we suppose that the solution is smooth enough. By multiplying \eqref{eq: flow} by $u$ and integrating in space and time, we get \eqref{eq: massConserv}. Moreover, by taking the scalar product of \eqref{eq: flow} with $\partial_t u$ and exploiting \eqref{eq: massConserv}, we get 
    \begin{equation}\label{eq: est1}
    \begin{aligned}
        \| \partial_t u(t) \|_{L^2}^2 &= -\frac{1}{2} \frac{d}{dt} \big\| \nabla u(t) \big\|_{L^2}^2 + \frac{\mu[u(t)]}{2\sigma + 2} \d \big\| u(t) \big\|_{L^{2\sigma +2}}^{2\sigma + 2}  \\
        &= -\frac{1}{2} \frac{d}{dt} \big\| \nabla u(t) \big\|_{L^2}^2 + \d \left(\frac{\mu[u(t)]}{2\sigma + 2}  \big\| u(t) \big\|_{L^{2\sigma +2}}^{2\sigma + 2}\right) - \frac{\big\| u(t) \big\|_{L^{2\sigma +2}}^{2\sigma + 2}}{2\sigma + 2}   \d \mu[u(t)]
    \end{aligned}
    \end{equation}
    The definition of $\mu[u]$ in \eqref{eq: mu} implies that the second term on the right-hand-side of \eqref{eq: est1} is given by
    \begin{equation*}
       \d \left(\frac{\mu[u(t)]}{2\sigma + 2}  \big\| u(t) \big\|_{L^{2\sigma +2}}^{2\sigma + 2}\right) = \frac{1}{2\sigma + 2}\d\| \nabla u(t) \|_{L^2}^2
    \end{equation*}
    Moreover, the last term is 
    \begin{equation*}
        \frac{\big\| u(t) \big\|_{L^{2\sigma +2}}^{2\sigma + 2}}{2\sigma + 2}   \d \mu[u(t)] = \frac{1}{2\sigma + 2}\d\| \nabla u(t) \|_{L^2}^2 - ( \| \nabla u(t)\|_{L^2}^2 + \omega \| u(t)\|_{L^2}^2) \d \ln \left( \| u(t) \|_{L^{2\sigma + 2}} \right).
    \end{equation*}
    So, we write \eqref{eq: est1} as 
    \begin{equation*}
        \frac{1}{2} \d \| \nabla u(t) \|_{L^2}^2 = - \| \partial_t u(t) \|_{L^2}^2 + ( \| \nabla u(t)\|_{L^2}^2 + \omega \| u(t)\|_{L^2}^2) \d \ln \left( \| u(t) \|_{L^{2\sigma + 2}} \right).
    \end{equation*}
    For $( \| \nabla u(t)\|_{L^2}^2 + \omega \| u(t)\|_{L^2}^2) = 0$, this implies \eqref{eq: KinDer} and otherwise we obtain
    \begin{equation*}
         \frac{1}{2} \frac{\d \left( \| \nabla u(t) \|_{L^2}^2 + \omega \| u(t) \|_{L^2}^2 \right)}{  \| \nabla u(t)\|_{L^2}^2 + \omega \| u(t)\|_{L^2}^2 } - \d \ln \left( \| u(t) \|_{L^{2\sigma + 2}} \right) = - \frac{ \| \partial_t u(t) \|_{L^2}^2}{ \| \nabla u(t)\|_{L^2}^2 + \omega \| u(t)\|_{L^2}^2}.
    \end{equation*}
    This implies
    \begin{equation*}
        \d \ln \left( \frac{\sqrt{ \| \nabla u(t)\|_{L^2}^2 + \omega \| u(t)\|_{L^2}^2}}{\| u (t) \|_{L^{2\sigma + 2}}} \right)  = - \frac{ \| \partial_t u(t) \|_{L^2}^2}{ \| \nabla u(t)\|_{L^2}^2 + \omega \| u(t)\|_{L^2}^2},
    \end{equation*}
    which is estimate \eqref{eq: estFundamental}.  
    A standard approximation technique can be used to obtain \eqref{eq: KinDer} and \eqref{eq: estFundamental} in our setting. 
\end{proof}

We now prove that \eqref{eq: massConserv} and \eqref{eq: estFundamental} yield a uniform bound on the $H^1$-norm of the solution. We start with the case $\omega \geq 0$. 

\begin{cor}\label{cor: boundness}
      For any $\omega \geq 0$ and $u_0 \in H^1_0(\Omega)$, the corresponding solution to \eqref{eq: flow} is global in time and there exists  $C = C(u_0) >0$ such that  
          \begin{equation}\label{eq: UnifBnd1}
              \| u \|_{L^\infty([0,\infty), H^1)} \leq C.
          \end{equation}
    Furthermore, for $\omega >0$ and any $u_0 \not\equiv 0$,   there exists  $K = K(u_0) >0$ such that for any $t \in [0,\infty)$, 
        \begin{equation}\label{eq: InfPotEn}
              \|  u(t) \|_{L^{2\sigma + 2}} \geq K.
        \end{equation} 
\end{cor}

\begin{proof}
    Combining \eqref{eq: massConserv}, \eqref{eq: estFundamental} and \eqref{eq: GN}, we obtain 
    \begin{equation}\label{eq: InfPotComp}
        \begin{aligned}
            \| \nabla u(t) \|_{L^2}^2 & \leq  \| \nabla u(t) \|_{L^2}^2 + \omega \| u_0 \|_{L^2}^2 \leq F[u_0] \| u(t) \|_{L^{2\sigma + 2}}^2 \leq C_{GN}^2 F[u_0] \| \nabla u(t) \|_{L^2}^{\frac{d\sigma}{\sigma + 1}}.  
        \end{aligned}
    \end{equation}
    Since $\frac{d\sigma}{\sigma + 1} < 2$, this implies \eqref{eq: UnifBnd1}. Moreover, we get
    \begin{equation*}
        \frac{\omega \| u_0 \|_{L^2}}{F[u_0]}\leq  \frac{\| \nabla u(t) \|_{L^2}^2 + \omega \| u_0 \|_{L^2}}{F[u_0]}\ \leq  \| u(t)\|_{L^{2\sigma + 2}}^2.
    \end{equation*}
     Notice that for $\Omega \subset \R^d$ bounded, we also have the lower bound \eqref{eq: UnifBnPot} which yields
        \begin{equation*}
        \| u (t)\|_{L^{2\sigma + 2}} \geq \frac{\| u_0 \|_{L^2}}{Vol(\Omega)^{\frac{\sigma}{2\sigma + 2}}}.
    \end{equation*}
\end{proof}
 We proceed to obtain a uniform $H^1$-norm upper bound for $\omega < 0$.
    \begin{prop}\label{prp: bond2}
         For any $\omega < 0$ and $u_0 \in H^1_0(\Omega)$, 
         \begin{equation*}
              \mathbf{G} = \{ v \in H^1_0(\Omega), \| \nabla v \|_{L^2}^2 \leq - \omega \| u_0\|_{L^2}^2 \}.
         \end{equation*}
         and
         \begin{equation*}
             \mathbf{H} = \{ v \in H^1_0(\Omega), \| \nabla v \|_{L^2}^2 > - \omega \| u_0\|_{L^2}^2 \},
         \end{equation*}
         are invariant sets. Moreover, \eqref{eq: UnifBnd1} is true. If $u_0 \in \mathbf{H} $, then \eqref{eq: InfPotEn} holds.
    \end{prop}

    \begin{proof}
        Equation \eqref{eq: estFundamental} implies that $F[u_0] $ and $F[u(t)]$ have the same sign for any $t \in [0, T_{max})$. If $ u_0 \in \mathbf{H}$, then $F[u_0] > 0$ and by \eqref{eq: KinDer}, $F[u(t)] > 0$ for any $t>0$.    Thus $\mathbf{G}$ and $\mathbf{H}$ are invariant sets. 
        For $u_0 \in \mathbf{G}$ we obtain the uniform bound 
        \begin{equation*}
            \| \nabla u \|_{L^\infty([0,\infty),L^2)} \leq -\omega \| u_0 \|_{L^2}^2.
        \end{equation*}
       For  $u_0 \in \mathbf{H}$, using \eqref{eq: GN}, \eqref{eq: massConserv} and \eqref{eq: estFundamental}, we obtain
       \begin{equation*}
       \begin{aligned}
            \| \nabla u(t) \|_{L^2}^2 & \leq  F[u_0] \| u(t) \|_{L^{2\sigma + 2}}^2  - \omega \| u_0 \|_{L^2}^2 \\
            & \leq C_{GN}^2 F[u_0] \| u_0 \|_{L^2}^{2 - \frac{d\sigma}{\sigma + 1}}  \| \nabla u(t) \|_{L^2}^\frac{d\sigma}{\sigma + 1} - \omega \| u_0 \|_{L^2}^2.
       \end{aligned}
    \end{equation*}
    Since $2 > \frac{d\sigma}{\sigma + 1}$, we obtain \eqref{eq: UnifBnd1}. Moreover, \eqref{eq: InfPotEn} holds because of \eqref{eq: InfPotComp}.
    \end{proof}

\section{Asymptotic Behavior}\label{secAsympt}

We start with a result in bounded domains. Here we prove that one stationary state belongs to the $\omega$-limit set of a given initial condition.

\begin{lem}\label{lem: oneStSt}
    Let $\Omega \subset \R^d$ be bounded, $ \omega > 0$, and $0 < \sigma < \frac{2}{(d-2)^+}$. Let  $u_0 \in H^1_0(\Omega)$, and $u \in C([0, \infty), H^1_0(\Omega))$ be the corresponding solution to \eqref{eq: flow}. Then there exists $(t_j)_{j \in \N}$, $t_j \to \infty$ as $j \to \infty$ and $Q \in H^1_0(\Omega)$ satisfying \eqref{eq: stst} such that 
    \begin{equation}\label{eq: convStSt}
        \lim_{j \to \infty} \| u(t_j) - Q \|_{H^1} \to 0.
    \end{equation}
\end{lem}

\begin{proof}
The functional $F$ in \eqref{eq: lyapF} is continuous, decreasing in time, and has a lower bound. Indeed,  by \eqref{eq: GN} and \eqref{eq: massConserv} we get
\begin{equation*}
    F[u(t)] \geq \frac{\| \nabla u(t) \|_{L^2}^2 + \omega \| u_0 \|_{L^2}^2}{C_{GN}^2\| u_0 \|_{L^2}^{2 - \frac{d\sigma}{\sigma + 1}}\| \nabla u(t)\|_{L^2}^{\frac{d\sigma}{\sigma + 1}}}.
\end{equation*}  
So, from \eqref{eq: estFundamental}, we get
\begin{equation}\label{eq: FInfty}
    \lim_{t\to \infty} F[u(t)]  = F[u_0]   \exp \left(-2 \int_0^\infty \frac{\| \partial_s u(s)\|_{L^2}^2}{\|\nabla u(s) \|_{L^2}^2 + \omega \| u(s) \|_{L^2}^2} \, ds \right) = F_\infty >0,
\end{equation}
and
\begin{equation*}
    \int_0^\infty \frac{\| \partial_s u(s) \|_{L^2}^2}{\|\nabla u(s)\|_{L^2}^2 +  \omega \| u(s) \|_{L^2}^2} \, ds = \ln\left( \sqrt{\frac{F[u_0]}{F_\infty}}\right) < \infty.
\end{equation*}
For \eqref{eq: UnifBnd1}, there exists $M > 0$ such that
\begin{equation*}
    \sup_{t \geq 0} \| \nabla u(t)\|_{L^2} \leq M < 0.
\end{equation*}
It follows that 
\begin{equation*}
    \frac{1}{M^2 + \omega \| u_0\|_{L^2}^2} \int_0^\infty \| \partial_s u(s) \|_{L^2}^2 \, ds  \leq  \int_0^\infty \frac{\| \partial_s u(s) \|_{L^2}^2}{\|\nabla u(s)\|_{L^2}^2 +  \omega \| u_0 \|_{L^2}^2} \, ds  < \infty,
\end{equation*}
that is 
\begin{equation*}
    \int_0^\infty \| \partial_s u(s) \|_{L^2}^2 \, ds  < \infty.
\end{equation*}
Moreover, \eqref{eq: UnifBnPot} implies that  $\mu[u(t)]$ is bounded from above and below. 
Thus, there exists $(t_j)_{j \in \N}$, $t_j \to \infty$ as $j \to \infty$, $u_\infty \in H^1_0(\Omega)$ and $\mu_\infty \in \R^+$ such that $u(t_j) \rightharpoonup u_\infty$ in $H^1_0(\Omega)$, $\partial_t u(t_j) \to 0$ in $L^2(\Omega)$ and $\mu[u(t_j)] \to \mu_\infty$ as $j \to \infty$. From the compact embedding $H^1_0(\Omega) \hookrightarrow L^p (\Omega)$ for any $2 \leq p < \frac{2d}{(d-2)^+}$ we have $\| u_\infty \|_{L^2} = \| u_0 \|_{L^2}$ and  $\lim_{j\to \infty} \| u(t_j)  - u_\infty \|_{L^{2\sigma +2}} = 0$. Moreover, the asymptotic profile satisfies weakly 
\begin{equation*}
       0 = \Delta u_\infty - \omega u_\infty + \mu_\infty |u_\infty|^{2\sigma} u_\infty.
\end{equation*}
By taking the scalar product of this equation with $u_\infty$, we obtain that $\mu_\infty = \mu[u_\infty]$. So, we have that $\mu[u(t_j)] \to \mu[u_\infty]$ as $j \to \infty$. This implies that $\|  u(t_j) \|_{H^1} \to \|  u_\infty \|_{H^1}$ as $j \to \infty$. 
\end{proof}

The same proof works for $\Omega = \R^d$, $d \geq 2$ and $u_0 \in H^1_{rad}(R^d)$.

\begin{lem}\label{lem: oneStRad}
   Let $\omega > 0$, $u_0 \in H^1_{rad}(\R^d)$, and $u \in C([0, \infty), H^1_{rad}(\R^d))$ be the corresponding solution to \eqref{eq: flow}. Then there exists $(t_j)_{j \in \N}$, $t_j \to \infty$ as $j \to \infty$ and $Q \in H^1_0(\Omega)$ satisfying \eqref{eq: stst} such that 
    \begin{equation*}
        \lim_{j \to \infty} \| u(t_j) - Q \|_{H^1} \to 0.
    \end{equation*}
\end{lem}

\begin{proof}
We keep the same notations as in the proof of Lemma \ref{lem: oneStSt}.
   From the compact embedding $H^1 \hookrightarrow L^{2\sigma + 2}$ and from $\mu[u(t_j)] \to \mu[u_\infty]$ we obtain that 
    \begin{equation*}
        \| u(t_j) \|_{H^1} \to \| u_\infty \|_{H^1},
    \end{equation*}
    from the weak lower semicontinuity of the norms. Indeed 
   \begin{align*}
   	&\| \nabla u_\infty\|_{L^2}^2 \leq \liminf  \| \nabla u(t_j)\|_{L^2}^2, \\
   	&\|  u_\infty\|_{L^2}^2 \leq \liminf  \| u(t_j)\|_{L^2}^2 
   \end{align*}
    which immediately implies the strong $H^1$ convergence. 
     The result follows from the same argument as the proof of Lemma \ref{lem: oneStSt}.
\end{proof}

We give a characterization of the $\omega$-limit set, defined in \eqref{eq: omegaLimit}.

\begin{cor}
    In the hypothesis of Lemma \ref{lem: oneStSt} or \ref{lem: oneStRad}, $W(u_0)$ is connected and compact with respect to the $L^p$-norm for any $p \in [2, 2/(d-2)^+)$. Moreover, there exists $F_\infty >0$ such that  
    \begin{equation*}
        W(u_0) \subset \mathcal{S}_\infty
    \end{equation*}
    where 
    \begin{equation*}
        S_\infty = \{ v\in \mathcal{S}_0\, :  \, F[v] = F_\infty \},
    \end{equation*}
    where $F_\infty$ is defined in \eqref{eq: FInfty}
\end{cor}

\begin{proof}
    Let $v \in W(u_0)$ and let  $(t_j)_{j \in \N}$, $t_j \to \infty$ as $j \to \infty$ be such that $u(t_j) \to v $ in $H^1_0(\Omega)$. Then, from the lower semicontinuity of the $H^1$-norm and the monotonicity of $F[u(t)]$ we have that
    \begin{equation*}
        F[v] = \lim_{j \to \infty} F[u(t_j)] = \lim_{t \to \infty} F[u(t)] = F_\infty.
    \end{equation*}
    Assume that there exists $v_0 \in \omega (u_0) \setminus \mathcal{S}_0$. Let $(t_n)_{n \in \N} \subset \R^+$, $t_n \to \infty$ as $n\to \infty$ be such that $u(t_n) \to v_0$ in $H^1_0(\Omega)$. Let $v \in C([0,\infty),H^1_0(\Omega))$ be the solution to \eqref{eq: flow} stemming from $v_0$.
    Then Lemma \ref{lem: oneStSt} implies that there exists $(\tau_n)_{n\in \N} \subset \R^+$ be such that $v(\tau_n) \to w$ in $H^1_0(\Omega)$, where $w \in \mathcal S_0$. Then we have $u(t_n + \tau_n) \to w$ in $H^1_0(\Omega)$, thus $w \in W(u_0)$ and $F[w] = F[v_0] = F_\infty$. Thus, by \eqref{eq: estFundamental}, we have
    \begin{equation*}
        \int_0^\infty \frac{\| \partial_\tau v(\tau)\|_{L^2}^2}{\| \nabla v(\tau)\|_{L^2}^2 + \omega \| u(\tau) \|_{L^2}^2} \, d\tau = F[w] - F[v_0] = 0,
    \end{equation*}
    that is $ \partial_t v = 0$ almost everywhere in $L^2$. So $v$ satisfies \eqref{eq: stst} in a weak sense for almost every $t >0$. Since $v \in H^1_0(\Omega)$, we get $v \in \mathcal S_0$, which is a contradiction. 

    Finally, notice that for any $s>0$, the set 
    \begin{equation*}
        \{ u(t):  \, t\geq s\}
    \end{equation*}
    is connected and relatively compact in $L^p(\Omega)$ for any $p\in \left[2, \frac{2d}{(d-2)^+}\right)$. Thus, the set 
    \begin{equation*}
         W(u_0) = \bigcap_{s>0} \overline{\{u(t):  \, t\geq s\}}
     \end{equation*}
    is connected and compact in $L^p(\Omega)$.
\end{proof}

\begin{theorem}
		Let $0 < \sigma < \frac{2}{(d-2)^+}$, $\omega >0$, $\Omega = \{ x \in \R\, : \, |x| < R \}$ for some $R >0$, $u_0 \in H^1_0(\Omega)$, $u_0 \geq 0$, $u_0 \centernot{\equiv} 0$ and let $u \in C([0,\infty), H^1_0(\Omega))$ be the solution to \eqref{eq: flow}. Then 
		\begin{equation*}
			u(t) \rightarrow Q_{gs} \ \mbox{ in } \ H^1_0(\Omega)\ \mbox{ as } \ t\rightarrow \infty.
		\end{equation*}
  where $Q_{gs}$ is the unique positive solution to \eqref{eq: stst}.
\end{theorem}
	
\begin{proof}
        By Lemma \ref{lem: oneStRad}, the Maximum Principle  \cite[Ch. 2, Sec. 4]{Fr2013} and Proposition \ref{lemMinBND}, there exists 
		$\{t_n\}_{n\in \N}$, $t_n \rightarrow \infty$ as $n\to \infty$, such that $u(t_n) \rightarrow Q_{gs}$  in $H^1_0(\Omega)$ and $\lim_{n \to \infty} F(u(t_n)) = F(Q_{gs})$.  Since $F$ is a continuous and decreasing in time, we obtain  
		\begin{equation}\label{eq:grt}
			\lim_{t \rightarrow \infty} F[u(t)] = \lim_{n \rightarrow \infty} F[u(t_n)] = F[Q_{gs}].
		\end{equation}
		Suppose, by contradiction, that $u(t) \rightarrow Q$ in $H^1_0(\Omega)$ is not true. Then there exists $\{t_k\}_{k\in \N}$ such that $\| u(t_k)\|_{L^2} = \| u_0\|_{L^2},$ $F[u(t_k)] \rightarrow F[Q_{gs}]$ and there exists $\varepsilon > 0$ so that 
		\begin{equation*}
		 \| u(t_k) - Q_{gs}\|_{H^1} \geq \varepsilon \ \mbox{for all} \ k \in \N.
		\end{equation*}
		 Since $\sup_{k\in\N}\|u(t_k)\|_{H^1} \leq C < \infty$, there exists a subsequence, still denoted by $\{t_k\}_{k\in \N}$ and a profile $u \in H^1_0(\Omega)$ so that $u(t_k) \rightharpoonup u$ in $H^1_0(\Omega)$. By the compact embedding $H^1_0(\Omega) \hookrightarrow L^2(\Omega) 
         \cap L^{2\sigma +2}(\Omega)$, we obtain $\|u(t_k)\|_{L^{2\sigma + 2}} \rightarrow \|u\|_{L^{2\sigma + 2}}$. By the lower semi-continuity, it follows that
		\begin{equation*}
			E[u] \leq \liminf_k E[u(t_k)] = E[Q],
		\end{equation*}
		which implies that $E[u] = E[Q]$. In particular, $\| \nabla u(t_k)\|_{L^2} \rightarrow \|\nabla Q\|_{L^2}$, and so $u(t_k) \rightarrow Q$ in $H^1_0(\Omega)$. This is a contradiction with $\| u(t_k) - Q\|_{H^1} \geq \varepsilon$. 
\end{proof}
	
\bibliographystyle{abbrv} 
\bibliography{biblio}
\end{document}